\documentclass[10pt]{amsart}

\usepackage{amsmath, amsthm, amscd, amsfonts, amssymb, graphicx, color}
\usepackage{latexsym}
\usepackage{amscd}
\usepackage{amsthm}
\usepackage{amssymb}
\usepackage{enumerate}
\usepackage{mathabx}

\theoremstyle{plain}
\newtheorem{theorem}{Theorem}
\newtheorem{corollary}[theorem]{Corollary}
\newtheorem{lemma}[theorem]{Lemma}

\theoremstyle{remark}

\newcommand{\C}{\mathbb{C}}
\newcommand{\R}{\mathbb{R}}
\newcommand{\N}{\mathbb{N}}
\newcommand{\Z}{\mathbb{Z}}

\renewcommand*{\span}{\ensuremath{\mathrm{span\,}}}

\begin{document}

\title{Montel--type theorems for exponential polynomials}
\author{J.~M.~Almira$^*$, L.~Sz\'ekelyhidi}


\thanks{$^*$ Corresponding author }
\let\thefootnote\relax\footnotetext{The research was supported by the Hungarian National Foundation for Scientific Research (OTKA),   Grant No. NK-81402.}

\address{Departamento de Matem\'{a}ticas, Universidad de Ja\'{e}n, E.P.S. Linares,  C/Alfonso X el Sabio, 28, 23700 Linares, Spain}
\email{jmalmira@ujaen.es}
\address{Institute of Mathematics, University of Debrecen, Egyetem t\'er 1, 4032 Debrecen, Hungary --- Department of Mathematics, University of Botswana, 4775 Notwane Rd. Gaborone, Botswana }
\email{lszekelyhidi@gmail.com}

\maketitle
\begin{abstract}
In this paper we characterize local exponential monomials and polynomials on different types of Abelian groups and we prove Montel--type theorems for these function classes.  
\end{abstract}

\section{Notation and preliminaries}

In this paper $\N,\Z,\R,\C$ denote the set of natural numbers, integers, reals, and complex numbers, respectively. We note that $0$ is in $\N$. We use the following standard multi-index notation: for each natural number $d$ the elements of $\N^d$ are called {\it multi-indices}. Whenever $\alpha, \beta$ are multi-indices, and $x$ is in $\C^d$, then we write
\begin{eqnarray*}
|\alpha|&=&\alpha_1+\alpha_2+\dots+\alpha_d\\
\alpha^{\beta}&=&\alpha_1^{\beta_1} \alpha_2^{\beta_2}\cdots \alpha_d^{\beta_d}\\
x^{\alpha}&=&x_1^{\alpha_1}x_2^{\alpha_2}\cdots x_d^{\alpha_d}\,.
\end{eqnarray*}
We use the convention $0^0=1$. We note that the last equation extends to $\alpha$ in $\Z^d$ assuming $x$ is in $\C_0^d$, where $\mathbb{C}_0$ denotes the set of nonzero complex numbers.
\vskip.3cm

The function $f:\C^d\to \C$ is called an {\it exponential}, if it is a continuous homomorphism of the additive group of $\C^n$ into the multiplicative group of nonzero complex numbers. The function $f:\C^d\to \C$ is called an {\it exponential monomial}, if it is the product of an exponential and a polynomial. Linear combinations of exponential monomials are called {\it exponential polynomials}. Hence the general form of exponential polynomials on $\C$ is the following: for each $z$ in $\C^d$ we have
\begin{equation*}
f(z)=\sum_{|\alpha|\leq N} a_{\alpha} z^{\alpha} e^{\langle \lambda, z\rangle}\,,
\end{equation*}
where $N$ is a nonnegative integer, $a_{\alpha}$ is a complex number for each $|\alpha|\leq N$, further $\lambda$ is in $\C^d$, and $\langle \lambda, z\rangle$ is the inner product in $\C^d$. In particular, the restrictions of exponential polynomials to $\Z^d$ can be written in the form
\begin{equation*}
f(n)=\sum_{|\alpha|\leq N} a_{\alpha} n^{\alpha} \lambda^n
\end{equation*}
for each $n$ in $\Z^d$, where $\lambda$ is in $\C_0^d$.
\vskip.3cm

These concepts have natural extensions to any topological Abelian group in place of $\C^d$. On any topological Abelian group $G$ we use the term {\it exponential} for a continuous complex homomorphism of $G$ into the multiplicative topological group of nonzero complex numbers. However, the concept of "polynomial" can be generalized in several different ways. One depends on the concept of {\it additive function}, which is a continuous homomorphism of $G$ into the additive topological group of complex numbers. A {\it polynomial} on $G$ is a function having the form $x\mapsto P\bigl(a_1(x),a_2(x),\dots,a_k(x)\bigr)$, where $P:\C^k\to\C$ is a complex polynomial in $k$ variables, which we shall call in this paper an {\it ordinary polynomial}, and $a_1,a_2,\dots,a_k$ are additive functions. Finally, we call a function an {\it exponential monomial}, if it is the product of a polynomial and an exponential. In this case, if the polynomial is nonzero, then the exponential is unique, and we say that $f$ {\it corresponds to} the exponential in question. 
\vskip.3cm

 The other concept of polynomial is related to Fr\'echet's functional equation
\begin{equation}\label{Frech1}
\Delta_{y_1,y_2,\dots,y_{n+1}} f(x)=0\,,
\end{equation}
where $n$ is a natural number, $x,y_1,y_2,\dots,y_{n+1}$ are in $G$, and $\Delta_y$ stands for the {\it difference operator} defined by
\begin{equation*}
\Delta_y f(x)=f(x+y)-f(x)
\end{equation*}
for each $x,y$ in $G$ and function $f:G\to\C$, further $\Delta_{y_1,y_2,\dots,y_{n+1}}$ denotes the product
\begin{equation*}
\Delta_{y_1,y_2,\dots,y_{n+1}} = \Delta_{y_1}\circ \Delta_{y_2}\circ\dots\circ \Delta_{y_{n+1}}\,.
\end{equation*}
Sometimes the functional equation
\begin{equation}\label{Frech2}
\Delta_y^{n+1} f(x)=0\,,
\end{equation}
where $n$ is a natural number, $x,y$ are in $G$, is also called Fr\'echet's equation. Here $\Delta_y^{n+1}$ denotes the $n+1$-th iterate of $\Delta_y$. It turns out that \eqref{Frech1} and \eqref{Frech2} are equivalent for complex valued functions on any Abelian group, as it has been proved in \cite{MR0265798} (see also \cite{Sze14b}).  
\vskip.3cm
 
The function $f:G\to\C$ is called a {\it generalized polynomial}, if it satisfies \eqref{Frech1} for each $x,y_1,y_2,\dots,y_{n+1}$ in $G$, and it is called a {\it generalized exponential monomial}, if it is the product of a generalized polynomial and an exponential. Here the exponential is unique again, assuming that the generalized polynomial is nonzero. We use the same terminology as above, that is, linear combinations of exponential monomials, resp. generalized exponential monomials are called {\it exponential polynomials}, resp. {\it generalized exponential polynomials}. It is known that on finitely generated Abelian groups every generalized polynomial is a polynomial (see e.g. \cite{MR2167990, MR2968200}). It follows that on finitely generated Abelian groups every generalized exponential polynomial is an exponential polynomial. A function is called a {\it local polynomial},  a {\it local exponential monomial}, or a {\it local exponential polynomial}, if its restriction to every finitely generated subgroup is a polynomial, an exponential monomial, or an exponential polynomial, respectively.
\vskip.3cm

Given a topological Abelian group $G$ the set of all continuous complex valed functions on $G$ will be denoted by $\mathcal C(G)$. This space, equipped with the pointwise linear operations and with the topology of uniform convergence on compact sets is a locally convex topological vector space. If $G$ is discrete, then the corresponding topology is that of pointwise convergence. A subset of $\mathcal C(G)$ is called {\it translation invariant}, if with every element $f$ in this subset it also contains its {\it translate} $\tau_y f$ for each $y$ in $G$, where $\tau_y f(x)=f(x+y)$ for each $x,y$ in $G$. A closed translation invariant subspace in $\mathcal C(G)$ is called a {\it variety} on $G$. 
\vskip.3cm

A basic result on varieties on $G=\Z^d$ is the following (see \cite{MR0098951}).

\begin{theorem}\label{Lefr} (Lefranc, 1958) In each variety on $\Z^d$ the exponential monomials span a dense subspace.
\end{theorem}

An extension of Lefranc's Theorem \ref{Lefr} to finitely generated Abelian groups is obvious (see \cite[Theorem 2.23, Theorem 2.24]{MR2279454}).
It follows that every finite dimensional translation invariant linear space of complex valued functions on a finitely generated Abelian group consists of exponential polynomials. This theorem has the following generalization to topological Abelian groups (see e.g. \cite[Theorem 10.1.]{MR1113488}, p.~78).

\begin{theorem}\label{me}
Every finite dimensional translation invariant space of continuous complex valued functions on a topological Abelian group consists of exponential polynomials.
\end{theorem}

Another important contribution to the subject is the theorem of P.~M.~Anselone and J.~Korevaar in \cite{MR0169048}.

\begin{theorem}\label{AnsKor} (Anselone--Korevaar)
Every finite dimensional translation invariant space of continuous complex valued functions, or complex valued Schwartz distributions on the reals consists of exponential polynomials.
\end{theorem}

For the characterization of exponential monomials we use modified difference operators as they have been introduced in \cite{Sze14d} (see also \cite{Sze14f,  Sze14e, Sze14c}). The definition follows.
\vskip.3cm

Let $G$ be an Abelian group and let $f,\varphi:G\to\C$ be functions. For each $x,y$ in $G$ we define
\begin{equation*}
\Delta_{\varphi;y} f(x)=f(x+y)-\varphi(y) f(x)=\bigl(\tau_y-\varphi(y)\tau_0\bigr) f(x)\,.
\end{equation*}
Then $\Delta_{\varphi;y}$ is called {\it $\varphi$-modified difference operator} or simply {\it modified difference operator}. The higher order modified difference operators are defined in an obvious way, as the products
\begin{equation*}
\Delta_{\varphi;y_1,y_2,\dots,y_{n+1}} = \Delta_{\varphi;y_1}\circ \Delta_{\varphi;y_2}\circ\dots\circ \Delta_{\varphi;y_{n+1}}\,,
\end{equation*}
whenever $n$ is a natural number and $y_1,y_2,\dots,y_{n+1}$ are arbitrary in $G$. In case of $y=y_1,y_2,\dots,y_{n+1}$ we use the notation $\Delta_{\varphi;y}^{n+1}$ for the above product. We note that the translation operators $\tau_y$ obviously commute, hence in this notation the increments $y_1,y_2,\dots,y_{n+1}$ can arbitrarily be reordered.
\vskip.3cm

Modified difference operators can be used to characterize generalized exponential monomials and polynomials. For the details see \cite{Sze14e, Sze14c, Sze14d}. Here we need the following simple result.

\begin{theorem}\label{modchar}
Let $G$ be an Abelian group, $n$ a natural number and let $f,\varphi:G\to\C$ be functions. If $f$ is nonzero, and it satisfies
\begin{equation}\label{nmod}
\Delta_{\varphi;y_1,y_2,\dots,y_{n+1}} f(x)=0
\end{equation}
for each $x,y_1,y_2,\dots,y_{n+1}$ in $G$, then $\varphi$ is an exponential and $f$ is a generalized exponential monomial corresponding to $\varphi$.
\end{theorem}

\begin{proof}
We prove by induction on $n$. The case $n=0$ has been proved in \cite{Sze14d}. Suppose that $n\geq 1$ and $f\ne 0$ satisfies the above equation for each $x,y_1,y_2,\dots,y_{n+1}$ in $G$. If there exist elements $y_1,y_2,\dots,y_n$ in $G$ such that the function $g=\Delta_{\varphi;y_1,y_2,\dots,y_n} f$ is nonzero, then $\Delta_{\varphi;y} g(x)=0$ for each $x,y$ in $G$, which implies that $\varphi$ is an exponential, by the first part of the proof. On the other hand, if $g$ is identically zero for every choice of $y_1,y_2,\dots,y_n$ in $G$, then $\varphi$ is an exponential, by the induction hypothesis. 
\vskip.3cm

For the second statement we observe that
\begin{equation}\label{difmod}
\Delta_{\varphi;y_1,y_2,\dots,y_{n+1}} f(x)=\varphi(x+y_1+y_2+\dots+y_{n+1}) \Delta_{y_1,y_2,\dots,y_{n+1}} \bigl( f\cdot \widecheck{\varphi})(x)
\end{equation}
holds for each $x,y_1,y_2,\dots,y_{n+1}$ in $G$, which can be verified by easy calculation. Here $\widecheck{\varphi}$ is defined by $\widecheck{\varphi}(x)=\varphi(-x)$ for each $x$ in $G$. Equation \eqref{difmod} implies that \eqref{nmod} holds if and only if the function $f\cdot \widecheck{\varphi}$ satisfies Fr\'echet's functional equation \eqref{Frech1}, that is, $f=p\cdot \varphi$ with some generalized polynomial $p$.
\end{proof}

\section{A characterization of local exponential monomials}

In this section we prove a characterization theorem for local exponential monomials, which is based on the following theorem (see \cite[Theorem 2]{AlSz14}):

\begin{theorem} \label{TLP}
Let $G$ be an Abelian group. The function $f:G\to\C$ is a local polynomial if and only if for each positive integer $t$, and elements $g_1,g_2,\dots,g_t$ in $G$ there are natural numbers $n_i$ for $i=1,2,\dots,t$ such that
\begin{equation}
\Delta_{g_i}^{n_i+1}*f(x)=0\,
\end{equation}
holds for $i=1,2,\dots,t$ and for all $x$ in the subgroup generated by $g_1,g_2,\dots,g_t$.
\end{theorem}

Using this theorem we have the following result.

\begin{theorem} \label{TLEM}
Let $G$ be an Abelian group. The function $f:G\to\C$ is a local exponential  monomial if and only if there exists an exponential $m$ on $G$, and for each positive integer $t$, and elements $g_1,g_2,\dots,g_t$ in $G$ there are natural numbers $n_i$ for $i=1,2,\dots,t$ such that
\begin{equation}\label{Frech4}
\Delta_{m;g_i}^{n_i+1} f(x)=0
\end{equation}
holds for $i=1,2,\dots,t$ and for all $x$ in the subgroup generated by $g_1,g_2,\dots,g_t$.
\end{theorem}

\begin{proof}
To prove the sufficiency we apply the identity \eqref{difmod}, which implies that for each positive integer $t$, and elements $g_1,g_2,\dots,g_t$ in $G$ there are natural numbers $n_i$ for $i=1,2,\dots,t$ such that
\begin{equation}\label{Frech5}
\Delta_{g_i}^{n_i+1} \bigl(f\cdot \widecheck{m}\bigr)(x)=0
\end{equation}
holds for $i=1,2,\dots,t$, and for all $x$ in the subgroup generated by $g_1,g_2,\dots,g_t$. Indeed, exponentials never vanish, hence form \eqref{difmod} we immediately obtain \eqref{Frech5}. Applying Theorem \ref{TLP} we get that $f\cdot \widecheck{m}$ is a local polynomial, which implies our statement.
\vskip.4cm

For the proof of necessity we observe that if $f$ is a nonzero local exponential monomial on $G$, then for each finitely generated subgroup $H$ there exists an exponential $m_H$ on $H$ and a natural number $n_H$ such that 
\begin{equation*}
\Delta_{m_H;h}^{n_H+1} f(x)=0
\end{equation*}
holds for each $x,h$ in $H$. As $f$ is nonzero, there exists a finitely generated subgroup $H_0$ such that the restriction of $f$ to $H_0$ is nonzero. Let $\mathcal F$ denote the set of all finitely generated subgroups of $G$, which include $H_0$. Clearly, $\bigcup \mathcal F=G$. We define $m:G\to\C$ by
\begin{equation*}
m(x)=m_H(x)\,,
\end{equation*}
whenever $x$ is in $H$ with $H$ in $\mathcal F$. It is obvious that $m$ is well-defined on $G$. If $x,y$ are in $G$, then there is an $H$ in $\mathcal F$ such that $x,y$ are in $H$, and we have
\begin{equation*}
m(x+y)=m_H(x+y)=m_H(x) m_H(y)=m(x) m(y)\,,
\end{equation*} 
that is, $m$ is an exponential on $G$. We have proved that if $f$ is a local exponential monomial on $G$, then there exists an exponential $m$ on $G$, and for each finitely generated subgroup $H$ there exists a natural number $n_H$ such that 
\begin{equation*}
\Delta_{m;h}^{n_H+1} f(x)=0
\end{equation*}
holds for each $x,h$ in $H$, which proves the necessity of our condition and our proof is complete.
\end{proof}

In \cite{AlSz14} we also proved the following result.

\begin{theorem}\label{cor_dist}
Let $t$ be a positive integer, let $h_1,h_2,\dots,h_t$ be elements in $\R^d$ and let $n_1,n_2,\dots,n_t$ be natural numbers. Suppose that the complex valued distribution $u$ satisfies 
\begin{equation} \label{poli}
\Delta_{h_k}^{n_k+1}u=0 
\end{equation}
for $k=1,2,\dots,t$. If the vectors $h_1,h_2,\dots,h_t$ generate a dense subgroup in $\R^d$, then $f$ is an ordinary polynomial of degree at most $n_1+n_2+\cdots+n_t+t-1$.
In particular, generalized polynomials and local polynomials in distributional sense are ordinary polynomials.
\end{theorem}

Using this theorem we obtain the following result exactly in the same way as above.

\begin{theorem}\label{TLD}
Let $u$ be a complex valued distribution on $\R^d$ and $t$ a positive integer. If there exists an exponential $m$ on $\R^d$, and there are elements $g_1,g_2,\dots,g_t$ generating a dense subgroup in $\R^d$, further there are natural numbers $n_i$ for $i=1,2,\dots,t$ such that
\begin{equation}\label{Frech6}
\Delta_{m;g_i}^{n_i+1} u(x)=0
\end{equation}
holds for $i=1,2,\dots,t$ and for all $x$ in the subgroup generated by $g_1,g_2,\dots,g_t$, then $u$ is an exponential monomial.
\end{theorem}

We recall that if $f:\R^d\to\C$ is a function, then $\tau_h$ is defined by
$$
\tau_hf(x)=f(x+h)\,,
$$
and if $u$ is a complex valued distribution on $\R^d$, then 
$$
(\tau_hu)(\phi)=u(\tau_{-h}\phi)\,,
$$
where $x,h$ are in $\R^d$ and $\phi$  is an arbitrary test function. Then $\tau_h u$ is again a distribution. Consequently, we have 
$$
\Delta_{\varphi;h} u=\bigl(\tau_{h}-\varphi(h) \tau_0\bigr) u
$$
for each distribution $u$, function $\varphi:G\to\C$ and element $h$ in $\R^d$. Then the meaning of $\Delta_{\varphi;h_1,h_2,\dots,h_s} u$ is obvious, too. Finally, when we claim that the distribution $u$ is an exponential polynomial, we mean that there exists an exponential polynomial $f:\mathbb{R}^d\to \mathbb{C}$ such that $u=f$ in distributional sense. In particular, $u$ is a locally integrable function and $u(x)=f(x)$ almost everywhere. 
\vskip.3cm

We summarize our results in the following corollary.
 
 \begin{corollary} \label{coro1L}
We suppose that $t$ is a positive integer, and either of the following possibilities holds:
\begin{enumerate}
\item $G$ is a finitely generated Abelian group with generators $h_1,\dots,h_t$, and \hbox{$f:G\to\C$}  is a function.
\item $G$ is a topological Abelian group, in which the elements $h_1,\dots,h_t$ generate a dense subgroup in $G$, and  $f:G\to\C$ is a continuous function.
\item $G=\mathbb{R}^d$,  the elements $h_1,\dots,h_t$ generate a dense subgroup, and $f$ is a complex valued distribution on $\R^d$.
\end{enumerate}
If there are natural numbers $n_1,n_2,\dots,n_t$, and there is an exponential $m:G\to\C$ such that $f$ satisfies 
 \[
 \Delta_{m;h_k}^{n_k}f=0 
 \]
 for $k=1,2,\dots,t$, then $f$ is an exponential monomial corresponding to the exponential function $m$.
\end{corollary}
 
 \begin{proof}
Case $(1)$ follows directly from Theorem \ref{TLEM}. Case $(2)$ follows from Theorem \ref{TLEM} and Case $(3)$ is Theorem \ref{TLD}.  
 \end{proof}
 
\section{Subspaces which are $\Delta_{m;y}^s$-invariant}

A weak point in Corollary \ref{coro1L} is that we have to assume that $m$ is an exponential. In the subsequent sections we shall weaken this hypothesis.

\begin{lemma}\label{NN} Given a vector space $V$ of complex valued functions on $G$,  and a function $\varphi:G\to\C$ the following statements are equivalent:
\begin{enumerate}
\item $V$ is translation invariant, that is $\tau_yf$ is in $V$ for each $y$ in $G$ and $f$ in $V$.
\item $V$ is difference invariant, that is $\Delta_yf$ is in $V$ for each $y$ in $G$ and $f$ in $V$.
\item $V$ is $\varphi$-modified difference invariant, that is $\Delta_{\varphi;y}f$ is in $V$  for each $y$ in $G$ and $f$ in $V$. 
\end{enumerate}

Consequently, if $y_1,\cdots,y_t$ generate $G$, and $\Delta_{\varphi;y_k}(V)\subseteq V$ for $k=1,\cdots,t$, then $V$ is translation invariant. Analogously, if $G$ is a topological Abelian group, $y_1,\cdots,y_t$ generate a dense subgroup of $G$ and $V\subseteq C(G,\mathbb{C})$, then $V$ is invariant by translations. In particular, if $V$ is finite dimensional, then and all its elements are exponential polynomials.
\end{lemma}

\begin{proof}
Indeed, if $f$ belongs to $V$, then the function $g$ defined by 
$$
g(x)=\Delta_{\varphi;y}f(x)=f(x+y)-\varphi(y)f(x)
$$ 
belongs to $V$ if and only if $h(x)=\tau_{y}f(x)=f(x+y)$ is in $V$, since $V$ is a vector space.
\end{proof}

Obviously, Lemma \ref{NN} has an analogous version for the distributional setting.

\begin{lemma}\label{nuevonuevo_lem} Given a vector space $V$ of complex valued distributions on $\R^d$, and a function $\varphi:G\to\C$  the following statements are equivalent:
\begin{enumerate}
\item $V$ is translation invariant, that is $\tau_yu$ is in $V$ for each $y$ in $G$ and $u$ in $V$.
\item $V$ is difference invariant, that is $\Delta_yu$ is in $V$ for each $y$ in $G$ and $u$ in $V$.
\item $V$ is $\varphi$-modified difference invariant, that is $\Delta_{\varphi;y}u$ is in $V$  for each $y$ in $G$ and $u$ in $V$.  
\end{enumerate}
Consequently, if $y_1,\cdots,y_t$ generate $\mathbb{R}^d$ and $\Delta_{\varphi;y_k}(V)\subseteq V$ for $k=1,\cdots,t$, then $V$ is translation invariant. In particular, if $V$ is finite dimensional, then all its elements are exponential polynomials. 
\end{lemma}

\begin{proof}
For the first part of this result the proof of Lemma \ref{NN} applies. The last statement is Anselone-Koreevar's Theorem \ref{AnsKor}.
\end{proof}

Given a vector space $E$, a subset $V\subseteq E$, a linear operator $L:E\to E$, and a natural number $n$ we introduce the notation
\begin{equation*}
V_L^{[n]}=V+L(V)+\dots +L^n(V)\,.
\end{equation*}
As $L^0$ is the identity operator, we have $V_L^{[0]}=V$.
\vskip.3cm  

Using a slightly different notation the following  technical result has been proved in \cite[Lemma 2.1]{AK_CJM}:

\begin{lemma} \label{uno} Let $E$ be a vector space, $L:E\to E$ a linear operator, and let $n$ be a positive integer. If $V$ is an $L^n$-invariant subspace of $E$, then the linear space $V_L^{[n]}$
is $L$-invariant. Furthermore, $V_L^{[n]}$ is the smallest $L$-invariant subspace of $E$ containing $V$.
\end{lemma}

\begin{proof}
Let $v$ be in $V_L^{[n]}$, then 
\begin{equation}\label{Box}
v=v_0+Lv_1+\dots+L^{n-1}v_{n-1}+L^nv_n
\end{equation}
with some elements $v_0,v_1,\dots,v_n$ in $V$. By the $L^n$-invariance of $V$, we have that $L^nv_n=u$ is in $V$, hence it follows
\begin{equation*}
Lv=L(v_0+u)+L^2v_1+\dots+L^{n}v_{n-1}\,,
\end{equation*}
and the right hand side is clearly in $V_L^{[n]}$. This proves that $V_L^{[n]}$ is $L$-invariant. On the other hand, if $W$ is an $L$-invariant subspace of $E$, which contains $V$, then $L^k(V)\subseteq W$ for $k=1,2,\dots,n$, hence the right hand side of \eqref{Box} is in $W$.
\end{proof}

Now we can prove the following result, which generalizes  \cite[Lemma 2.2]{AK_CJM}:

\begin{lemma}\label{dos1}
Let $t$ be a positive integer, $E$ a vector space, $L_1,L_2,\cdots,L_t:E\to E$ pairwise commuting linear operators, and let $s_1,\cdots, s_t$ be natural numbers. Given a subspace $V\subseteq E$ we form the sequence of subspaces
\begin{equation}\label{rec}
V_0=V,\enskip V_{i}=(V_{i-1})_{L_{i}}^{[s_{i}]},\enskip i=1,2,\dots,t\,.
\end{equation}
If  for $i=1,2,\dots,t$ the subspace $V$ is $L_i^{s_i}$-invariant, then $V_t$ is $L_i$-invariant, and it contains $V$. Furthermore, $V_t$ is the smallest subspace of $E$ containing $V$, which is $L_i$-invariant for $i=1,2,\dots,t$.
\end{lemma}

\begin{proof}
First we prove by induction on $i$ that $V_{i}$ is $L_{j}^{s_{j}}$-invariant and it contains $V$ for each $i=0,1,\dots,t$ and $j=1,2,\dots,t$. For $i=0$ we have $V_0=V$, which is  $L_j^{s_j}$-invariant for $j=1,2,\dots,t$, by assumption. 
\vskip.3cm

Suppose that $i\geq 1$, and we have proved the statement for $V_{i-1}$. Now we prove it for $V_{i}$. If $v$ is in $V_{i}$, then we have
\begin{equation*}
v=u_0+L_{i}u_1+\dots+L_{i}^{s_{i}}u_{s_i}\,,
\end{equation*}
where $u_j$ is in $V_{i-1}$ for $j=0,1,\dots,s_{i}$. It follows for $j=1,2,\dots,t$
\begin{equation*}
L_{j}^{s_{j}}v=(L_{j}^{s_{j}}u_0)+L_{i}(L_{j}^{s_{j}}u_1)+\dots+L_{i}^{s_{i}}(L_{j}^{s_{j}}u_{s_i})\,.
\end{equation*}
Here we used the commuting property of the given operators, which obviously holds for their powers, too. By the induction hypothesis, the elements in the brackets on the right hand side belong to $V_{i-1}$, hence $L_{j}^{s_{j}}v$ is in $V_i$, that is, $V_i$ is $L_{j}^{s_{j}}$-invariant. As $V_i$ includes $V_{i-1}$, we also conclude that $V$ is in $V_i$, and our statement is proved.
\vskip.3cm

Now we have
\begin{equation*}
V_t=V_{t-1}+L_t(V_{t-1})+\dots +L_t^{s_t}(V_{t-1})\,,
\end{equation*}
and we apply the previous lemma: as $V_{t-1}$ is $L_t^{s_t}$-invariant, we have that $V_t$ is $L_t$-invariant. 
\vskip.3cm

Let us now prove the invariance of $V_t$
under the operators $L_j$ ($j<t$). Since $V_1$
is clearly $L_1$-invariant, by Lemma 12, an induction process 
gives that $V_{t-1}$ is $L_i$-invariant for $1\leq i\leq t-1$. Thus, if we
take $1\leq i\leq t-1$, then we can use that $L_iL_t=L_tL_i$  and $L_i(V_{t-1})$ is a subset
of $V_{t-1}$ to conclude that
\begin{eqnarray*}
L_i(V_t)
&=& L_i(V_{t-1})+L_t(L_i(V_{t-1}))+\cdots +L_t^{s_t}(L_i(V_{t-1})) \\
&\subseteq & V_{t-1}+L_t(V_{t-1})+\cdots+L_t^{s_t}(V_{t-1}) =V_t\,,
\end{eqnarray*}
which completes this part of the proof. 
\vskip.3cm

Suppose that $W$ is a subspace in $E$ such that $V\subseteq W$, and $W$ is $L_j$-invariant for $j=1,2,\dots,t$. Then, obviously, all the subspaces $V_i$ for $i=1,2,\dots,t$ are included in $W$. In particular, $V_t$ is included in $W$. This proves that $V_t$ is the smallest subspace in $E$, which includes $V$, and which is invariant with respect to the family of operators $L_i$. Consequently, it follows that $V_t$ is uniquely determined by $V$, and by the family of the operators $L_i$, no matter how we label these operators. 
\end{proof}

 \begin{theorem} \label{coro2}
We suppose that $t$ is a positive integer, and either of the following possibilities holds:
\begin{enumerate}
\item $G$ is a finitely generated Abelian group with generators $h_1,\dots,h_t$, and $V$ is a finite dimensional vector space of complex valued functions on $G$.
\item $G$ is a topological Abelian group, in which the elements $h_1,\dots,h_t$ generate a dense subgroup, and $V$ is a finite dimensional vector space of continuous complex valued functions on $G$.
\item $G=\mathbb{R}^d$,  the elements $h_1,\dots,h_t$ generate a dense subgroup in $G$, and $V$ is a finite dimensional vector space of complex valued distributions on $\R^d$.
\end{enumerate}
If there are natural numbers $n_1,n_2,\dots,n_t$, and there is a function $\varphi:G\to\C$ such that $\Delta_{\varphi;h_k}^{n_k+1}(V)\subseteq V$ holds for $k=1,2,\dots,t$, then 
$V$ is included in a finite dimensional translation invariant linear space. In particular, $V$ consists of exponential polynomials.
\end{theorem}

\begin{proof} We apply Lemma \ref{dos1} with $L_i=\Delta_{\varphi;h_i}:E\to E$, $i=1,\cdots,t$, to conclude that, with the notation $W=V_t$, we have
$V\subseteq W$, and $W$ is a finite dimensional subspace satisfying $\Delta_{\varphi;h_i}(W)\subseteq W$, $i=1,2,\cdots,t$. The hypotheses on $\{h_1,\cdots,h_t\}$ and Lemmas \ref{NN}, \ref{nuevonuevo_lem} imply that $W$ is translation invariant. Theorems \ref{Lefr}, \ref{me}, and \ref{AnsKor} guarantee that $W$ consists of exponential polynomials. In particular, $V$ consists of exponential polynomials, too. 
\end{proof}

 

 \section{A Monte--type theorem for exponential monomials}

  \begin{theorem} \label{montel_all_cases}
  We suppose that $t$ is a positive integer, and either of the following possibilities holds:
\begin{enumerate}
\item $G$ is a finitely generated Abelian group with generators $h_1,\dots,h_t$, and \hbox{$f:G\to\C$} is a nonzero function.
\item $G$ is a topological Abelian group, in which the elements $h_1,\dots,h_t$ generate a dense subgroup, and $f:G\to\C$ is a nonzero continuous function.
\item $G=\mathbb{R}^d$,  the elements $h_1,\dots,h_t$ generate a dense subgroup in $G$, and $f$ is a nonzero complex valued distribution on $\R^d$.
\end{enumerate}
If there are natural numbers $n_1,n_2,\dots,n_t$, and there is a function $\varphi:G\to\C$ such that
 \begin{equation*}
 \Delta_{\varphi;h_k}^{n_k+1}f=0,\enskip k=1,\dots,t,
 \end{equation*}
 then  $f$ is an exponential monomial.  Furthermore, if $e:G\to \mathbb{C}$ is the exponential function associated to $f$, then $m(h_i)=e(h_i)$, $i=1,\cdots,t$.
  Consequently, if  $\Delta_{\varphi;y}^{n+1}f(x)=0$ for all $x,y$ in $G$, then  $\varphi=e$.  
  \end{theorem}

 \begin{proof} Let us prove $(1)$.  The other cases  are direct consequences of this one. 
It follows from Theorem \ref{coro2} that $f$ is an exponential polynomial:
 \begin{equation} \label{efe}
 f(x)=\sum_{i=1}^sp_i(x) e_i(x)\,,
 \end{equation}
 where $e_1,\cdots,e_s:G\to\mathbb{C}$ are exponentials  with $e_i\neq e_j$ whenever $i\neq j$, further $p_1,\cdots,p_s:G\to \mathbb{C}$ are polynomials, of degrees $k_1,k_2,\cdots,k_s$, respectively.   We can assume, with no loss of generality,  that $k_i\geq s_i-1$ for all $i$ and 
 \[
 p_i(x)=\sum_{|\alpha|\leq k_{i}}c_{i,\alpha} a(x)^{\alpha}
 \]
 where $a(x)=(a_1(x),a_2(x),\cdots,a_d(x))$, and the functions $a_1,\cdots,a_d:G\to \mathbb{C}$ are additive and linearly independent, further $c_{i,\alpha}$ is a complex number for each multi-index $\alpha\in\mathbb{N}^d$. 
Under such conditions it is known (see \cite{MR1113488}, Lemma 4.8, p.~44) that the functions in the set
 \begin{equation} \label{base}
 B = \{
  a^{\alpha}e_k, \ \  0\leq |\alpha|\leq k_i \text{ and } i=1,2,\cdots,s\} 
 \end{equation}
 are linearly independent, hence they form a basis of the generated vector space, which we denote by  $\mathcal{E}=\span B$.  Furthermore, $f$ is in $\mathcal{E}$, by construction.  
 \vskip.3cm
 
We consider the linear map $\Delta_{\varphi;h}:\mathcal{E}\to\mathcal{E}$ induced by the operator $\Delta_{\varphi;h}$, when restricted to $\mathcal{E}$.  Obviously, $\mathcal{E}= E_1\oplus E_2\oplus \cdots\oplus E_s$, where 
 \[
 E_j=\span\{a^{\alpha}e_i\} _{0\leq |\alpha|\leq k_i},  \  j=1,2,\cdots,s. 
 \]
 Furthermore,  $\Delta_{\varphi;h}(E_j)\subseteq E_j$ for $j=1,2,\cdots,s$, since
 \[
 q_{\alpha,h}(x)=\Delta_ha(x)^{\alpha}=(a(x)+a(h))^\alpha-a(x)^\alpha
 \]
 is a polynomial of degree at most $|\alpha|-1$, and 
 \begin{eqnarray*}
 \Delta_{\varphi;h}(a(x)^{\alpha}e_k(x)) &=& ((a(x)+a(h))^\alpha e_k(h)-m(h)a(x)^\alpha)e_k(x) \\
 &=& (a(x)^\alpha( e_k(h)-m(h))+q_{\alpha,h}(x)e_k(h))e_k(x).
 \end{eqnarray*}
 It follows that for each $h$ in $G$ and  $p\geq 1$ the operator $\Delta_{\varphi;h}^p$ also satisfies the relation $\Delta_{\varphi;h}^p(E_j)\subseteq E_j$ whenever $j=1,2,\cdots,s$, hence if $g$ is in $\mathcal{E}$, then 
 $\Delta_{\varphi;h}^p(g)=0$ if and only if  $\Delta_{\varphi,h}^pb_j=0$, where $g=b_1+\cdots+b_s$, with  $b_j$ in $E_j$ for $j=1,2,\cdots,s$. 
 \vskip.3cm
 
 Let $j$ be in $\{1,\cdots,s\}$, and we consider the restriction of the operator $\Delta_{\varphi;h}$ to $E_j$, denoting it by the same symbol. We order the basis $B_j= \{a^{\alpha}e_j\} _{0\leq |\alpha|\leq k_j}$ of $E_j$ by the {\it graded lexicographic order}: 
 \[
 a^{\alpha}e_j\leq_{grlex} a^{\gamma}e_j
 \]
 if and only if
 \[
 |\alpha|\leq |\gamma| \text{ or } (|\alpha|=|\gamma| \text{ and } \alpha\leq_{lex} \gamma) \,,
 \] 
 where $\leq_{lex}$ refers to the {\it lexicographic order}. The matrix $A_j$ associated to the operator $\Delta_{\varphi;h}$ with respect to this basis
 is upper triangular, and the entries in its main diagonal are all equal to  $d_j(h)=e_j(h)-\varphi(h)$. Obviously, this implies that if $\varphi(h)\neq e_j(h)$, then $d_j(h)\neq 0$, so that the restriction of $\Delta_{\varphi;h}$ to $E_j$ is invertible. In particular, if $\Delta_{\varphi;h}b_j=0$ for some $b_j$ in $E_j$, then $b_j=0$.
 \vskip.3cm 
 
 Let $f=b_1+\cdots+b_s$ be in $\mathcal{E}$ such that  $\Delta_{\varphi;h_j}^{s_j}f=0$ for all $j$. Then $\Delta_{\varphi;h_{j}}^{s_{j}}b_i=0$ for all $1\leq i\leq s$ and all $1\leq j\leq t$. Thus, if $b_{i_1},b_{i_2}\neq 0$ with $i_1\neq i_2$ (so that $f$ is not an exponential monomial) and we take $j\in\{1,\cdots,t\}$, then $\Delta_{\varphi;h_{j}}^{s_{j}}b_{i_1}=0$ with $b_{i_1}\neq 0$ implies that $\varphi(h_j)=e_{i_1}(h_j)$. The same argument, when applied to $b_{i_2}$, shows that $\varphi(h_j)=e_{i_2}(h_j)$. Hence $e_{i_1}(h_j)=e_{i_2}(h_j)$ for all $1\leq j\leq t$. As $\{h_1,\cdots,h_t\}$ generates $G$ and $e_{i_1},e_{i_2}$ are homomorphisms, this implies $e_{i_1}=e_{i_2}$, which is a contradiction. It follows that $f$ is an exponential monomial. Furthermore, we have also proved that if $f=b_{k_0}$, then $\varphi(h_j)=e_{k_0}(h_j)$ for all $j$.  This ends the first part of the proof. 
\vskip.3cm
 
 If we assume that $\Delta_{\varphi;y}^{n+1}f(x)=0$ for all $x,y$ in $G$, then we can add this $y$ to the system $\{h_i\}_{i=1}^t$ to get a generating set of $G$ with $t+1$ elements, and then we apply the result to infer $\varphi(y)=e_{k_0}(y)$. This proves that $\varphi=e_{k_0}$ is an exponential.  
 \end{proof}
 
 The following corollary is evident.
 
 \begin{corollary}Let $G$ be an Abelian group and assume that $\Delta_{\varphi;h}^{n+1}f=0$ admits a nonzero solution $f$. Then $\varphi$ is an exponential on $G$. 
 \end{corollary}

\section{Montel--type theorem for exponential polynomials and a characterization of local exponential polynomials }
 
 Let $G$ be an Abelian group, and let $H$ be the subgroup of $G$ generated by the elements $\{g_1,\cdots,g_t\}$ of $G$, further let  $\{n_{i,k}\}_{1\leq i\leq r, 1\leq k\leq t}$ be a finite set of natural numbers.  We say that  the complex valued functions $\{\varphi_k\}_{k=1}^r$ on $H$ form a {\it minimal set} of functions for the functional equation 
 \begin{equation}\label{Frech3330}
 \Delta_{\varphi_1;g_{i_1}}^{n_{1,i_1}+1} \Delta_{\varphi_2;g_{i_2}}^{n_{2,i_2}+1} \cdots \Delta_{\varphi_r;g_{i_r}}^{n_{r,i_r}+1}f(x)=0\,
 \end{equation}
with $1\leq i_k\leq t$, $k=1,\cdots,r$, if $f:G\to\C$ is a function satisfying equation \eqref{Frech3330}  for each $x$ in $H$, and either $r=1$ and $f|_H\neq 0$, or $r\geq 2$ and for each $k$ in $\{1,\cdots,r\}$, there exist  natural numbers $1\leq a_j\leq t$ and $j$ in $\{1,\cdots,k-1,k+1,\cdots,t\}$ such that 
 \[
 \Delta_{\varphi_1;g_{a_1}}^{n_{1,a_1}+1} \cdots  \Delta_{\varphi_{k-1};g_{a_{k-1}}}^{n_{k-1,a_{k-1}}+1} \Delta_{\varphi_{k+1};g_{a_{k+1}}}^{n_{k+1,a_{k+1}}+1}\cdots 
 \Delta_{\varphi_r,g_{a_r}}^{n_{r,a_r}+1}f(x_0) \neq 0 
\] 
for some $x_0$ in $H$.  
 
 \begin{theorem} \label{teo_Expo}
Suppose that one of the following cases holds:
 \begin{enumerate}
 \item $G$ is an Abelian group generated by the subset $\{g_1,\cdots,g_{t}\}$, and 
 $f:G\to\C$ is a nonzero function. 
 \item $G$ is a topological Abelian group in which the subset $\{g_1,\cdots,g_{t}\}$ generates a dense subgroup, and  $f:G\to\C$ is a nonzero continuous function.
 \item $G=\mathbb{R}^d$, in which the subset $\{g_1,\cdots,g_{t}\}$ generates a dense subgroup, and $f$ is a nonzero complex valued distribution.
 \end{enumerate}
 Suppose moreover that there exist natural numbers $\{n_{i,k}\}_{1\leq i\leq r, 1\leq k\leq t}$ and functions $\varphi_k:G\to\C$ such that \eqref{Frech3330} holds for $1\leq i_k\leq t$, $k=1,\cdots,r$, and for each $x$ in $G$.
 \vskip.3cm

 Then the following statements hold:
 \begin{enumerate}[(i)]
 \item $f$ is an exponential polynomial of the form
 \begin{equation*}\label{decomposition_f}
 f=\sum_{i=1}^N p_i e_i\,,
 \end{equation*}
 where $p_i$ is a polynomial, and $e_i$ is an exponential for  $i=1,\cdots, N$. 
 \item If $\{\varphi_k\}_{k=1}^r$ is a minimal set of functions for the functional equation \eqref{Frech3330}, then there exist exponential 
 functions $m_k:G\to\C$ such that $m_k(g_i)=\varphi_k(g_i)$ for all $i=1,2,\dots,t$ and $k=1,2,\dots,r$. Moreover, if the functions $\varphi_k$ also  satisfy the equation
 \[
 \Delta_{\varphi_1;y_1}^{n_1+1} \Delta_{\varphi_2;y_{2}}^{n_{2}+1} \cdots \Delta_{\varphi_r;y_{r}}^{n_{r}+1}f(x)=0
 \]
for all $x,y_1,\cdots,y_r$ in $G$, then $\varphi_k$ is an exponential function for $k=1,2,\dots,r$, further $N\leq r$ and, possibly by renumbering  the $m$'s, we have $e_k(g_i)=m_k(g_i)$ for all $1\leq i\leq t$, $1\leq k\leq N$.
 \end{enumerate} 
 \end{theorem}

 \begin{proof} We prove $(1)$ as the other cases  are direct consequences of this statement.  
 First we prove $(i)$. By induction, we show that if $f$ satisfies \eqref{Frech3330} with some functions $\varphi_k:G\to\C$ ($k=1,2,\dots,r$), then $f$ is an exponential polynomial. This claim has already been proved for $r=1$, so that we assume $r\geq 2$ and we let $h_i=\Delta_{\varphi_1;g_i}^{n_{1,i}+1}f$. Then 
 \begin{equation}
 \Delta_{\varphi_2;g_{i_2}}^{n_{2,i_2}+1} \cdots \Delta_{\varphi_r;g_{i_r}}^{n_{r,i_r}+1}h_i(x)= \Delta_{\varphi_1;g_{i}}^{n_{1,i}+1} \Delta_{\varphi_2;g_{i_2}}^{n_{2,i_2}+1} \cdots \Delta_{\varphi_r;g_{i_r}}^{n_{r,i_r}+1}f(x)=0
 \end{equation}
 for $1\leq i,i_k\leq t$, $k=2,\cdots,r$, and for all $x$ in $G$. Thus, the induction hypothesis implies that $h_i$ is an exponential polynomial for $1\leq i\leq t$. 
In particular, $W_{i}=\tau(h_{i})$ is finite dimensional, for $1\leq i\leq t$. 
\vskip.3cm

Let $V=\span\{f\}+W_1+W_2+\cdots+W_t$.  Then $V$ is a finite dimensional space. Moreover,
 \[
 \Delta_{\varphi_1;g_i}^{n_{1,i}+1}(V)\subseteq V, \ \ i=1,2\cdots,t
 \]
since $h_i=\Delta_{\varphi_1;g_i}^{n_{1,i}+1}f$ is in $W_i \subseteq V$, $i=1,\cdots,t$, and, on the other hand,  for each $j$, the space $W_j$ is translation invariant.  Indeed, Lemma \ref{NN} implies that $W_j$ is  $\Delta_{m_1;g_i}$-invariant for each $i$, hence  it is also $\Delta_{\varphi_1;g_i}^{n_{1,i}+1}$-invariant for each $i$.
\vskip.3cm
 
 It follows from Theorem \ref{coro2}  that all elements of $V$ are exponential polynomials. In particular, $V\subseteq \bigoplus_{k=1}^NE_k$, where each $E_k$ is a finite dimensional  translation invariant vector space,  whose elements are exponential monomials with associated exponential function $e_k:G\to\mathbb{C}$, and $i\neq j$ implies $e_i\neq e_j$.  Thus $f$ is an exponential polynomial, and it can be decomposed as a sum
 \begin{equation*}
 f(x)=p_1(x)e_1(x)+\cdots+p_N(x)e_N(x)
 \end{equation*}
 where $p_i\,e_i$ is a nonzero exponential monomial in $E_i$ ($i=1,\cdots, N$). 
\vskip.3cm

Now we prove $(ii)$.  The minimality assumption on $\{\varphi_k\}_{k=1}^r$ implies that  for each $k$ in $\{1,\cdots,r\}$ there exist  natural numbers $a_j$ in the set $\{1,2,\dots,t\}$ and $j$ in the set $\{1,\cdots,k-1,k+1,\cdots,t\}$ such that 
 \[
 \Delta_{\varphi_1;g_{a_1}}^{n_{1,a_1}+1} \cdots  \Delta_{\varphi_{k-1};g_{a_{k-1}}}^{n_{k-1,a_{k-1}}+1} \Delta_{\varphi_{k+1};g_{a_{k+1}}}^{n_{k+1,a_{k+1}}+1}\cdots 
 \Delta_{\varphi_r;g_{a_r}}^{n_{r,a_r}+1}f(x) \neq 0 
\] 
for some $x_0$ in $G$.  Thus, we can  apply Theorem \ref{montel_all_cases} to 
$$
\psi_k(x)= \Delta_{\varphi_1;g_{a_1}}^{n_{1,a_1}+1} \cdots  \Delta_{\varphi_{k-1};g_{a_{k-1}}}^{n_{k-1,a_{k-1}}+1} \Delta_{\varphi_{k+1};g_{a_{k+1}}}^{n_{k+1,a_{k+1}}+1}\cdots 
 \Delta_{\varphi_r;g_{a_r}}^{n_{r,a_r}+1}f(x)\,,
$$
 since $\Delta_{\varphi_k;g_i}^{n_{k,i}+1}\psi_k(x)=0$ for all $x$ in $G$ and $i=1,\cdots,t$, further $\psi_k\neq 0$. Thus $\psi_k$ is an exponential monomial with associated exponential function $m_k$, and $m_k(g_i)=\varphi_k(g_i)$ for $1\leq i\leq t$.  Furthermore, if 
\begin{equation*} 
 \Delta_{\varphi_1;g_{a_1}}^{n_{1,a_1}+1} \cdots  \Delta_{\varphi_{k-1};g_{a_{k-1}}}^{n_{k-1,a_{k-1}}+1} \Delta_{\varphi_k;y_k}^{n_k+1}  \Delta_{\varphi_{k+1};g_{a_{k+1}}}^{n_{k+1,a_{k+1}}+1}\cdots 
 \Delta_{\varphi_r;g_{a_r}}^{n_{r,a_r}+1}f(x) =\Delta_{\varphi_k;y_k}^{n_k+1}  \psi_k(x)=0 
\end{equation*} 
for all $x,y_k$ in $G$, then $\varphi_k=m_k$ is an exponential function.  
\vskip.3cm

Now we prove that under the given conditions $N\leq r$, and each exponential $e_i$ interpolates one of the functions $\varphi_1,\varphi_2,\cdots,\varphi_r$ at the set of nodes $\{g_1,g_2,\cdots,g_t\}$.
\vskip.3cm
 
The computations in the proof of Theorem \ref{montel_all_cases} show that if $\varphi_i(g_j)\neq e_{k_0}(g_j)$, then the operator $\Delta_{\varphi_i;g_j} :E_{k_0}\to E_{k_0}$ is invertible. 
Assume that $e_{k_0}$ is such that for each $1\leq i\leq r$ the function $e_{k_0}$ does not interpolate $\varphi_i$ at the nodes $\{g_1,g_2,\cdots,g_t\}$. 
This means that for each $i$ in $\{1,\cdots,r\}$ there exists $a_i$ in $\{1,\cdots,t\}$ such that $e_{k_0}(g_{a_i})\neq \varphi_i(g_{a_i})$, so that $\Delta_{\varphi_i;g_{a_i}}^{n_{i,a_i}+1}:E_{k_0}\to E_{k_0}$ is invertible.   Thus, if the term $p_{k_0} e_{k_0}$ is nonzero in the decomposition of $f$ as a sum of exponential monomials, then
 \begin{eqnarray*}
 0 &=& \Delta_{\varphi_1;g_{a_1}}^{n_{1,a_1}+1} \Delta_{\varphi_2;g_{a_2}}^{n_{2,a_2}+1} \cdots \Delta_{\varphi_r;g_{a_r}}^{n_{r,a_r}+1}(f) \\
 &=& \sum_{k=1}^N\Delta_{m_1;g_{a_1}}^{n_{1,a_1}+1} \Delta_{\varphi_2;g_{a_2}}^{n_{2,a_2}+1} \cdots \Delta_{\varphi_r;g_{a_r}}^{n_{r,a_r}+1}(p_ke_k).
 \end{eqnarray*}
 Here the $k$-th term belongs to $E_k$ and the sum is a direct sum, hence it vanishes if and only if all the terms are zero. However  
 \[
 \Delta_{\varphi_1;g_{a_1}}^{n_{1,a_1}+1} \Delta_{\varphi_2;g_{a_2}}^{n_{2,a_2}+1} \cdots \Delta_{\varphi_r;g_{a_r}}^{n_{r,a_r}+1}(p_{k_0}e_{k_0})\neq 0
 \]
 since $p_{k_0}e_{k_0}\neq 0$, and the operator $\Delta_{\varphi_1;g_{a_1}}^{n_{1,a_1}+1} \Delta_{\varphi_2;g_{a_2}}^{n_{2,a_2}+1} \cdots \Delta_{\varphi_r;g_{a_r}}^{n_{r,a_r}+1}:E_{k_0}\to E_{k_0}$ is invertible.  This is a contradiction,  consequently we conclude that there exists $1\leq i=i(k_0)\leq r$ such that $e_{k_0}(g_j)=\varphi_i(g_j)$ for all $1\leq j\leq t$. As this holds for $1\leq k_0\leq N$, the proof is complete. 
  \end{proof}
 
  \begin{theorem} \label{cor_Expo} 
 Let $G$ be an Abelian group. The function $f:G\to\C$ is a local exponential  polynomial if and only if for each positive integer $t$, and for each elements $g_1,g_2,\dots,g_t$ in $G$ there are exponential functions $m_k:H\to\C$ defined on the subgroup $H$ generated by the $g_i$'s, and there are natural numbers $n_{i,k}$ for $1\leq i\leq r$, $1\leq k\leq t$ such that
 \begin{equation}\label{Frech33}
 \Delta_{m_1;g_{i_1}}^{n_{1,i_1}+1} \Delta_{m_2;g_{i_2}}^{n_{2,i_2}+1} \cdots \Delta_{m_r;g_{i_r}}^{n_{r,i_r}+1}f(x)=0\,
 \end{equation}
 holds for $1\leq i_k\leq t$, $k=1,\cdots,r$, and for all $x$ in $H$.
 \end{theorem}
 
\begin{proof}
If $f$ is a local exponential polynomial and $H$ is a finitely generated subgroup of $G$ with generators $g_1,\cdots,g_t$, then there exist polynomials $p_k:H\to\C$ and exponentials $m_k:H\to\mathbb{C}$ for $k=1,\cdots, r$ such that $f(x)=\sum_{k=1}^rp_k(x)m_k(x)$ for each $x$ in $H$. Then equation \eqref{Frech33} trivially holds for appropriate values of $n_{k,j}$. This proves the necessity of the condition.
\vskip.3cm

The sufficiency is a direct consequence of Theorem \ref{teo_Expo}.
\end{proof}
 
If $f:G\to\mathbb{C}$ is a local exponential polynomial, then the number of exponentials appearing in the decomposition of the restriction of $f$ to $H$ may depend on $H$. Indeed,  we can give the following example: let $G$ be the set of finitely supported complex sequences $x=(x_i)_{i\in\N}$,  and let $ f(x)=\sum_{i\in \N}(2^{i})^{x_i}x_i^i$.

\end{document}